%% file: ramsey.tex
\documentclass{amsart}

\usepackage{graphicx,tikz,amssymb,mathtools,tabularx,algorithm,enumerate,enumitem}
\usetikzlibrary{arrows,positioning,calc}
\usepackage[justification=centering]{caption}
\usepackage[noend]{algpseudocode}

\tikzset{
    vertex/.style={circle, fill=black, inner sep=1.6pt},
    newvertex/.style={circle, draw=black, fill=yellow, inner sep=1.6pt},
    blueedge/.style={blue, very thick},
    rededge/.style={red, very thick},
    candidate/.style={gray, dashed, thick}
}

\usetikzlibrary{decorations.pathreplacing}

\makeatletter
\def\BState{\State\hskip-\ALG@thistlm}
\makeatother

\newtheorem{theorem}{Theorem}[section]
\newtheorem{lemma}[theorem]{Lemma}
\newtheorem{prop}[theorem]{Proposition}
\newtheorem{corr}[theorem]{Corollary}

\author{Sam Beilis}
\address{Department of Mathematical Sciences, Kean University, Union, NJ 07083}
\curraddr{Department of Mathematics, University of Chicago, Chicago, IL 60637}
\email{sbeilis@uchicago.edu}

\author{Israel R. Curbelo}
\address{Department of Mathematical Sciences, Kean University, Union, NJ 07083}
\email{israel.curbelo@kean.edu}

\author{Elizabeth R. Koizumi}
\address{Department of Mathematics, Cornell University, Ithaca, NY 14853}
\email{erk87@cornell.edu}

\title[Online Ramsey Numbers of Stars versus Long Paths and Cycles]{Asymptotic Bounds for Online Ramsey Numbers of Stars versus Long Paths and Cycles}

\begin{document}

\begin{abstract}

The online Ramsey game for graphs $G$ and $H$ is played on the infinite complete graph $K_\mathbb{N}$. In each round, Builder chooses an edge, and Painter colors it red or blue. The online Ramsey number $\tilde{r}(G,H)$ is the smallest integer $t$ for which Builder has a strategy guaranteeing a red copy of $G$ or a blue copy of $H$ within $t$ rounds. For every fixed integer $k\ge4$, the best-known asymptotic lower bounds for $\tilde{r}(K_{1,k},P_n)$ and $\tilde{r}(K_{1,k},C_n)$ are $\left(\frac{k+3}{4}+o(1)\right)n$ as $n\to\infty$. We improve the corresponding asymptotic upper bounds from $\left(k+o(1)\right)n$ to $\left(\frac{k+3}{3}+o(1)\right)n$ as $n\to\infty$.
\end{abstract}

\maketitle

\input{introduction}

\input{proof}

\input{concluding}

\bibliographystyle{acm}
\bibliography{ramsey}

\end{document}

%% file: introduction.tex
\section{Introduction}

The online Ramsey problem for graphs $G$ and $H$ is defined as a two-player game between Builder and Painter. The game is played in rounds on the infinite complete graph $K_\mathbb{N}$. Each round, Builder chooses an edge and Painter, immediately and irrevocably, colors the edge red or blue. The \emph{online Ramsey number} $\tilde{r}(G,H)$ is the smallest integer $t$ for which Builder has a strategy that guarantees a red copy of $G$ or a blue copy of $H$ in at most $t$ rounds, regardless of the choices that Painter makes.
The online Ramsey number $\tilde{r}(G,H)$ is always bounded above by the \emph{size Ramsey number} $\hat{r}(G,H)$ introduced by Erd\H{o}s, Faudree, Rousseau, and Schelp \cite{efrs-78} which is the smallest integer $t$ for which there exists a graph $F$ with $t$ edges such that every 2-edge-coloring of $F$ with colors red and blue results in a red copy of $G$ or a blue copy of $H$. However, this bound is often far from optimal. 

The study of online Ramsey numbers involving stars, paths, and cycles
was initiated by Grytczuk, Kierstead, and Pra{\l}at \cite{gry-kie-pra-08}, who
obtained general bounds for $\tilde{r}(K_{1,k},H)$. In particular, their
results imply
\[
\frac{1}{2}\left\lfloor\frac{n}{2}\right\rfloor(k-1)+n-1
\le
\tilde{r}(K_{1,k},P_n)
\le
k(n-2)+1
\]
and
\[
\frac{k+3}{4}n
\le
\tilde{r}(K_{1,k},C_n)
\le
kn.
\]
Consequently, for fixed $k$ and $n\to\infty$,
\[
\left(\frac{k+3}{4}+o(1)\right)n
\le
\tilde{r}(K_{1,k},P_n),\,
\tilde{r}(K_{1,k},C_n)
\le
\left(k+o(1)\right)n.
\]
Thus, even for a fixed star, the previous bounds leave a substantial gap
between the lower and upper bounds.

Subsequent work has determined the asymptotic behavior, and in fact the
exact values, in the first nontrivial case $k=3$. Grytczuk, Kierstead,
and Pra{\l}at \cite{gry-kie-pra-08} showed that
\[
\tilde{r}(K_{1,3},P_n)
\ge
\left\lfloor\frac{3n}{2}\right\rfloor-1.
\]
Latip and Tan \cite{lat-tan-21} later proved
\[
\tilde{r}(K_{1,3},P_n)
\le
\frac{5}{3}n+O(1)
\]
and conjectured that
\[
\tilde{r}(K_{1,3},P_n)
=
\left(\frac{3}{2}+o(1)\right)n.
\]
This conjecture was resolved by Song, Wang, and Zhang \cite{son-wan-zha-25},
who proved that
\[
\tilde{r}(K_{1,3},P_n)
=
\left\lfloor\frac{3n}{2}\right\rfloor
\qquad\text{for all } n\ge 2.
\]

More recently, Zhi and Zhang \cite{zhi-zha-26} determined the corresponding
online Ramsey number for long cycles, proving that
\[
\tilde{r}(K_{1,3},C_n)
=
\left\lfloor\frac{3(n+1)}{2}\right\rfloor
\qquad\text{for all } n\ge 13.
\]
Hence, for $k=3$, both the path and cycle versions have asymptotic
coefficient $3/2$, showing that the general lower bound of Grytczuk,
Kierstead, and Pra{\l}at is asymptotically sharp in this case.

For fixed $k\ge 4$, however, the exact asymptotic behavior remains unknown. In particular,
the general bounds above leave a gap between the coefficient
$(k+3)/4$ in the lower bound and the coefficient $k$ in the upper
bound for both
\[
\tilde{r}(K_{1,k},P_n)
\qquad\text{and}\qquad
\tilde{r}(K_{1,k},C_n).
\]
Determining the asymptotic behavior of these quantities for general
fixed $k$ therefore remains a natural problem.

Beilis and Curbelo \cite{bei-cur-26} showed that for each
fixed $k$, the path and cycle problems have the same asymptotic behavior.
Consequently, any asymptotic upper bound for $\tilde{r}(K_{1,k},P_n)$ yields the
same asymptotic upper bound for $\tilde{r}(K_{1,k},C_n)$.

\begin{prop}\label{prop:limit}
    For every positive integer $k$, there is an $L_k>0$ such that \[\lim_{n\to\infty}\frac{\tilde{r}(K_{1,k},P_n)}{n}=L_k=\lim_{n\to\infty}\frac{\tilde{r}(K_{1,k},C_n)}{n}.\]
\end{prop}

In this paper, we improve the asymptotic upper bound of the online Ramsey number $\tilde{r}(K_{1,k},P_n)$. 
\begin{theorem}\label{mainthm}
    For every fixed positive integer $k$, as $n\to\infty$,
    \[
    \tilde{r}(K_{1,k},P_n)
    \le
    \left(\frac{k+3}{3}+o(1)\right)n.
    \]
\end{theorem}

As a consequence of Proposition \ref{prop:limit}, we also improve the asymptotic upper bound of the online Ramsey number $\tilde{r}(K_{1,k},C_n)$. 
\begin{corr}\label{maincorr}
    For every fixed positive integer $k$, as $n\to\infty$,
    \[\tilde{r}(K_{1,k},C_{n})\le \left(\frac{k+3}{3}+o(1)\right)n.\]
\end{corr}

The remainder of the paper is devoted to the proof of Theorem~\ref{mainthm}. In Section~2, we develop several Builder strategies for constructing and extending long blue paths while controlling the red degrees of their vertices. These ingredients are then combined in an iterative argument that yields the stated asymptotic bound.

%% file: proof.tex
\section{Proof of Main Theorem}

We begin with three auxiliary results that will be used in the proof of our main theorem. The first shows that two disjoint blue paths can be joined at a cost of at most $3k-1$ rounds. This significantly improves upon the approach in \cite{bei-cur-26} which required $5k$ rounds.

\newcommand{\kk}{K_{1,k}}

\begin{figure}[ht]
\centering
\input{fig2.tikz}
\caption{The merging strategy in Lemma~\ref{lem:fus}. If Builder obtains
blue edges $vw$ and $uw$ for some new vertex $w$, then the two blue paths
are joined to form a blue copy of $P_{m+n+1}$.}
\label{fig:fusion}
\end{figure}

\begin{lemma}\label{lem:fus}
Let $m$, $n$, and $k$ be positive integers, and suppose that Builder has forced a blue copy of $P_n\sqcup P_m$. Then Builder can force either a red copy of $\kk$ or a blue copy of $P_{m+n+1}$ in at most $3k-1$ additional rounds.
\end{lemma}

\begin{proof}
Let $v$ and $u$ be end vertices of the blue copies of $P_n$ and $P_m$, respectively, and let $V$ be a set of $2k-1$ isolated vertices. Builder chooses all edges between $v$ and $V$. If at least $k$ of these edges are colored red, then they form a red copy of $\kk$ centered at $v$. Thus, we may assume that at least $k$ of these edges are blue. Let $U\subseteq V$ be a set of $k$ vertices joined to $v$ by blue edges.

Builder now chooses all edges between $u$ and $U$. If all $k$ of these edges are colored red, then they form a red copy of $\kk$ centered at $u$. Otherwise, some edge $uw$, with $w\in U$, is blue. Since $vw$ is also blue, the two blue paths are joined through $w$, forming a blue copy of $P_{m+n+1}$.

In total, Builder uses at most
\[
(2k-1)+k=3k-1
\]
additional rounds.
\end{proof}

The next lemma gives a simple strategy for constructing a long blue path. In addition to the total number of rounds, we keep track of the red degrees of the vertices on the resulting path. This refined estimate will be useful when the average red degree is large.

\begin{lemma}\label{lem:path}
There is a Builder strategy that forces either a red copy of $\kk$ or a blue copy of $P_n$ in at most
\[
\frac{1}{2}(n+k-1)(k-1)+n-1
\]
rounds. Furthermore, if Builder obtains a blue copy of $P_n$ on vertices $v_1,\ldots,v_n$, and if $d_i$ denotes the red degree of $v_i$ for each $i\in[n]$, then the total number of rounds used is at most
\[
\frac{1}{2}\left[\sum_{i=1}^n d_i+(k-1)^2\right]+n-1.
\]
\end{lemma}

\begin{proof}
Let $V$ be a set of $n+k-1$ vertices, and choose $v_1\in V$ as the initial active vertex. Suppose that Builder has constructed a blue path
\[
P^{(i)}=v_1\ldots v_i,
\]
with active vertex $v_i$. In Phase $i$, Builder exposes edges from $v_i$ to vertices of $V\setminus V(P^{(i)})$ until Painter colors one of them blue.

If Painter colors $k$ such edges red, then Builder obtains a red copy of $\kk$ centered at $v_i$. Thus, assuming that no red copy of $\kk$ is created, Painter colors at most $k-1$ of these edges red. Since
\[
|V\setminus V(P^{(i)})|=n+k-1-i\ge k
\]
for $i\le n-1$, Builder obtains a blue edge $v_iv_{i+1}$ for some
$v_{i+1}\in V\setminus V(P^{(i)})$. Builder then sets
\[
P^{(i+1)}=v_1\ldots v_{i+1}
\]
and takes $v_{i+1}$ to be the new active vertex. Hence, if no red copy of $\kk$ is created, Builder constructs a blue copy of $P_n$ after $n-1$ phases.

Now suppose that a blue copy of $P_n$ is obtained, and let $d_i$ denote the red degree of $v_i$. The only blue edges exposed are the $n-1$ edges of the path. The remaining $k-1$ vertices of $V$ each have red degree at most $k-1$, since otherwise a red copy of $\kk$ would be created. Therefore, the sum of the red degrees over all vertices of $V$ is at most
\[
\sum_{i=1}^n d_i+(k-1)^2.
\]
By the handshake lemma, the number of red edges is at most
\[
\frac{1}{2}\left[\sum_{i=1}^n d_i+(k-1)^2\right].
\]
Thus, the total number of rounds is at most
\[
\frac{1}{2}\left[\sum_{i=1}^n d_i+(k-1)^2\right]+n-1.
\]
Since $d_i\le k-1$ for every $i\in[n]$, this is at most
\[
\frac{1}{2}(n+k-1)(k-1)+n-1.
\]
\end{proof}

As an immediate consequence, we obtain the following asymptotic upper bound which already improves the previously best known bound from \cite{gry-kie-pra-08}.

\begin{corr}\label{corr:1} 
For every fixed positive integer $k$, as $n\to\infty$, \[ \tilde{r}(\kk,P_n) \le \left(\frac{k+1}{2}+o(1)\right)n. \] 
\end{corr} 

\begin{proof} 
By Lemma~\ref{lem:path}, \[ \tilde{r}(\kk,P_n) \le \frac{1}{2}(n+k-1)(k-1)+n-1. \] Since $k$ is fixed as $n\to\infty$, \[ \frac{1}{2}(n+k-1)(k-1)+n-1 = \left(\frac{k+1}{2}+o(1)\right)n. \] 
\end{proof}

The preceding strategy becomes more efficient when the vertices of the blue path already have large red degree. The following lemma exploits this additional red structure to extend a blue copy of $P_n$ to a blue copy of $P_{2n}$.

\begin{figure}[ht]
\centering
\input{fig2-1.tikz}
\caption{The strategy in Lemma~\ref{lem:dub} efficiently doubles the length of the blue path from Lemma \ref{lem:path} when the average red degree of the vertices on the blue path is large.}
\label{fig:double}
\end{figure}

\begin{lemma}\label{lem:dub}
Let $n$ and $k$ be positive integers, and suppose
that Builder has forced a graph containing a set of
vertices $v_1,\ldots,v_n$ such that $v_i$ has red
degree $d_i$ for each $i\in[n]$.
There is a Builder strategy that forces either
a red copy of $K_{1,k}$ or a blue copy of
$P_{2n}$ using at most
\[
(k+1)n-\sum_{i=1}^n d_i+O(k^2)
\]
additional rounds.
\end{lemma}

\begin{proof}
If $n<k$, Builder can force a blue copy of
$P_{2n}$ in at most $k(2n-1)=O(k^2)$
rounds. Hence, we may assume that $n\ge k$.

Let $V=\{v_1,\ldots,v_n\}$, where $v_i$ has red
degree $d_i$. Let $X$ be a set of $n-k$
new isolated vertices.
We construct an auxiliary blue graph $B$, initially
consisting of the $n$ isolated vertices in $V$.
Throughout the construction, we maintain the
invariant that $B$ is a forest of blue paths.

For each $x\in X$, Builder proceeds as follows.
Let $b_1,\ldots,b_{k+1}$ be $k+1$ endpoints from $k+1$ distinct path components of $B$. Builder chooses the edge $xb_i$ for each $i\in[k+1]$ until Painter colors two of them blue. This is guaranteed since Painter cannot color $k$ edges incident to $x$ red without creating a red copy of $\kk$. Let $xb_p$ and $xb_q$ be the two blue edges. The path components in $B$ containing $b_q$ and $b_p$ are merged into a single blue path with an additional vertex $x$.  

Since $B$ initially has $n$ components and $|X|=n-k$, the resulting graph has $k$
blue path components.

Each phase produces exactly two blue edges, so
the total number of blue edges in this stage is
\[
2(n-k).
\] Meanwhile, every red edge chosen is incident
to some vertex $v_i\in V$. Since Painter cannot
create a red copy of $K_{1,k}$, at most
$k-1-d_i$ additional red edges in this stage can be incident
to $v_i$. Therefore, the total number of red
edges in this stage is at most
\[
\sum_{i=1}^n(k-1-d_i)
=(k-1)n-\sum_{i=1}^n d_i.
\]

By Lemma~2.1, Builder can merge two blue path
components in at most $3k-1$ additional rounds.
Since there are $k$ components remaining,
$k-1$ applications of Lemma~2.1 are
required. Hence, the final merging stage costs
at most
\(
(k-1)(3k-1)
\)
additional rounds. Each application of Lemma~2.1 adds one vertex
to the resulting blue path. Thus, after
merging all $k$ components, Builder obtains a
blue path on
\[
n+(n-k)+(k-1)=2n-1
\]
vertices. 

Lastly, by choosing edges from an endpoint of
this path to new isolated vertices, Builder can
extend it to a blue copy of $P_{2n}$ in at most
$k$ additional rounds.

Let $R$ denote the total number of
additional rounds. Then

\begin{align*}
R&\le (k-1)n-\sum_{i=1}^n d_i
  +2(n-k)+(k-1)(3k-1)+k\\
&=
(k+1)n-\sum_{i=1}^n d_i
  +(k-1)(3k-1)-k\\
  &=(k+1)n-\sum_{i=1}^n d_i+O(k^2).
\end{align*}

\end{proof}

\begin{figure}
    \centering
    \input{fig5.tikz}
    \caption{An outline for the proof of Theorem \ref{thm:iterative}.}
    \label{fig:dia}

\end{figure}

We are now ready to prove the main upper bound. 
The proof balances two strategies for extending a blue path (See Figure \ref{fig:dia}.) Suppose Builder has constructed a blue copy of $P_n$ whose vertices have average red degree $\mu$. The construction of this path uses approximately
\[
\left(1+\frac{\mu}{2}\right)n
\]
rounds. Builder then proceeds in one of two ways. If $\mu$ is large, Lemma \ref{lem:dub} allows Builder to extend the path directly to a blue copy of $P_{2n}$ using approximately
\[
(k+1-\mu)n
\]
additional rounds. If $\mu$ is small, Builder instead uses the current upper bound to construct a second blue copy of $P_n$ and then joins the two paths using Lemma \ref{lem:fus}. Thus, up to lower-order terms, the total number of rounds is bounded by
\[
\left(1+\frac{\mu}{2}\right)n
+
\min\left\{
(k+1-\mu)n,\,
\text{current inductive bound}
\right\}.
\]
Choosing a threshold for $\mu$ so that the two resulting bounds agree yields an improved coefficient. Iterating this procedure gives the following theorem.

\begin{theorem}\label{thm:iterative}
For every positive integer $k$ and every nonnegative integer $m$, as $n\to\infty$,
\[
\tilde{r}(\kk,P_n)
\le
\left[
\left(
\frac{1}{3}
+\frac{1}{6}\left(\frac{1}{4}\right)^m
\right)k
+
\left(
1-\frac{1}{2}\left(\frac{1}{4}\right)^m
\right)
\right]n+o(n).
\]
\end{theorem}

\begin{proof}
For $t\ge 0$, define
\[
\beta_t=
\frac{1}{3}
+\frac{1}{6}\left(\frac{1}{4}\right)^t
\qquad\text{and}\qquad
c_t=
1-\frac{1}{2}\left(\frac{1}{4}\right)^t,
\]
and observe that these sequences satisfy
\[
\beta_{t+1} = \frac{1}{4} + \frac{1}{4}\beta_t,
\qquad
c_{t+1} = \frac{3}{4} + \frac{1}{4}c_t.
\]

We prove the result by induction on $m$. The case $m=0$ follows from Corollary \ref{corr:1}.

Suppose now that the result holds for some $m\ge 0$. Thus, as $n\to\infty$, Builder can force either a red copy of $\kk$ or a blue copy of $P_n$ in at most
\[
(\beta_m k+c_m)n+o(n)
\]
rounds.

We show that the result also holds for $m+1$. Let $N$ be a positive integer and set
\[
n=\left\lceil\frac{N}{2}\right\rceil.
\]
Then
\[
2n\ge N
\qquad\text{and}\qquad
2n=N+O(1).
\]
We construct a Builder strategy that forces either a red copy of $\kk$ or a blue copy of $P_{2n}$ within
\[
(\beta_{m+1}k+c_{m+1}+o(1))(2n)
\]
rounds. Since $P_{2n}$ contains a copy of $P_N$, this will yield the desired bound for $P_N$.

The strategy consists of two phases.

In Phase 1, Builder first plays a strategy which either forces a red copy of $\kk$ in at most
\[
\frac{1}{2}(n+k-1)(k-1)+n-1
\]
rounds or a blue copy of $P_n$ on vertices $v_1,\ldots,v_n$ with red degrees $d_1,\ldots,d_n$ in at most
\[
\frac{1}{2}\left[\sum_{i=1}^n d_i+(k-1)^2\right]+n-1
\]
rounds. Since
\[
\frac{1}{2}(n+k-1)(k-1)+n-1
=
\left(\frac{k+1}{2}\right)n+o(n)
=
\left(\frac{k+1}{4}+o(1)\right)(2n),
\]
we may assume that the strategy results in the blue copy of $P_n$.
Let $\mu$ be the average red degree on the vertices of the blue path, that is,
\[
\mu=\frac{1}{n}\sum_{i=1}^n d_i.
\]

In Phase 2, we consider two cases, according to whether
$\mu>\alpha_mk$ or $\mu\le\alpha_mk$, where
\[
\alpha_m
=
1-\beta_m+\frac{1-c_m}{k}.
\]
Let $R_1$ and $R_2$ denote the total number of rounds in the first and second cases, respectively.

Suppose first that $\mu>\alpha_mk$. By Lemma \ref{lem:dub}, Builder can force either a red copy of $\kk$ or a blue copy of $P_{2n}$ in at most
\[
(k+1)n-\sum_{i=1}^n d_i+O(k^2)
\]
additional rounds. Then
\begin{align*}
R_1
&\le
\frac{1}{2}\left[\sum_{i=1}^n d_i+(k-1)^2\right]+n-1
+(k+1)n-\sum_{i=1}^n d_i+O(k^2)\\
&=(k+2)n-\frac{1}{2}\mu n+o(n)\\
&<(k+2)n-\frac{1}{2}\alpha_mkn+o(n)\\
&=2(\beta_{m+1}k+c_{m+1})n+o(n)\\
&=(\beta_{m+1}k+c_{m+1}+o(1))(2n).
\end{align*}

Suppose instead that $\mu\le\alpha_mk$. By the induction hypothesis, Builder can force either a red copy of $\kk$ or a blue copy of $P_n$ on a set of isolated vertices in at most
\[
(\beta_m k+c_m)n+o(n)
\]
rounds. If a red copy of $\kk$ is forced, then Builder wins and we are done. Hence, we may assume that a blue copy of $P_n$ is obtained. By Lemma \ref{lem:fus}, Builder can force either a red copy of $\kk$ or a blue copy of $P_{2n+1}$, which contains a blue copy of $P_{2n}$, in at most $3k-1$ additional rounds. Therefore,
\begin{align*}
R_2
&\le
\frac{1}{2}\left[\sum_{i=1}^n d_i+(k-1)^2\right]+n-1
+(\beta_m k+c_m)n+o(n)+3k-1\\
&\le
(\beta_m k+c_m+1)n+\frac{1}{2}\mu n+o(n)\\
&\le
(\beta_m k+c_m+1)n+\frac{1}{2}\alpha_mkn+o(n)\\
&=
\left[
\left(\frac{1}{2}\beta_m+\frac{1}{2}\right)k
+\frac{1}{2}c_m+\frac{3}{2}
\right]n+o(n)\\
&=2(\beta_{m+1}k+c_{m+1})n+o(n)\\
&=(\beta_{m+1}k+c_{m+1}+o(1))(2n).
\end{align*}

Thus, in either case, Builder can force either a red copy of $\kk$ or a blue copy of $P_{2n}$ within
\[
(\beta_{m+1}k+c_{m+1}+o(1))(2n)
\]
rounds. Since $2n\ge N$, the blue path contains a copy of $P_N$. Moreover,
\[
2n=N+O(1),
\]
and hence
\[
(\beta_{m+1}k+c_{m+1}+o(1))(2n)
=
(\beta_{m+1}k+c_{m+1})N+o(N).
\]
Therefore,
\[
\tilde r(\kk,P_N)
\le
(\beta_{m+1}k+c_{m+1})N+o(N),
\]
which completes the induction.
\end{proof}

Taking $m\to\infty$ in Theorem~\ref{thm:iterative}, we have \[ \beta_m\to\frac13 \qquad\text{and}\qquad c_m\to 1. \] Therefore, \[ \tilde{r}(\kk,P_n) \le \left(\frac{k+3}{3}+o(1)\right)n, \] which proves Theorem~\ref{mainthm}.

%% file: fig2.tikz
\begin{tikzpicture}[scale=1.8]

\tikzset{
    vertex/.style={circle, fill=black, inner sep=1.6pt},
    blueedge/.style={blue, very thick},
    rededge/.style={red, very thick},
    lightblueedge/.style={blue, opacity=.25, thin},
    lightrededge/.style={red, opacity=.25, thin}
}

\node[vertex,label=above:$v$] (v2) at (4,0) {};
\node[vertex,label=above:$u$] (v3) at (6,0) {};

\node (pn) at (2.75,0.2){\Large$P_n$};
\node (pm) at (7,0.2){\Large$P_m$};

\draw[blueedge] (2,0) node[left] {$\cdots$} -- (v2);
\draw[blueedge] (v3)--(7.5,0)node[right] {$\cdots$} ;

\node[newvertex] (b1) at (2.0,-1.8) {};
\node[newvertex] (b2) at (4,-1.8) {};


\foreach \x in {2.05,2.15,...,3.95}
    \draw[lightrededge] (v2) -- (\x,-1.8);

\draw[rededge] (v2)--(b1);
\draw[rededge] (v2)--(b2);

\draw[decorate,decoration={brace,amplitude=5pt}]
(4,-2) -- (2,-2)
node[midway,below=5pt] {$ \le k-1$};

\node[newvertex] (c1) at (4.4,-1.8) {};
\node[newvertex] (c2) at (6.3,-1.8) {};
\node[newvertex,label=below:$w$] (c3) at (6.7,-1.8) {};
\node[newvertex] (c4) at (7.3,-1.8) {};

\foreach \x in {4.45,4.55,...,6.25}
    \draw[lightblueedge] (v2) -- (\x,-1.8);
\foreach \x in {4.45,4.55,...,6.25}
    \draw[lightrededge] (v3) -- (\x,-1.8);
\foreach \x in {6.75,6.85,...,7.25}
    \draw[lightblueedge] (v2) -- (\x,-1.8);

\draw[blueedge] (v2)--(c1);
\draw[blueedge] (v2)--(c2);
\draw[blueedge] (v2)--(c3);
\draw[blueedge] (v2)--(c4);

\draw[rededge] (v3)--(c1);
\draw[rededge] (v3)--(c2);
\draw[blueedge] (v3)--(c3);

\draw[decorate,decoration={brace,amplitude=5pt}]
(6.3,-2) -- (4.4,-2)
node[midway,below=5pt] {$ \le k-1$};

\draw[decorate,decoration={brace,amplitude=5pt}]
(7.3,-2.5) -- (4.4,-2.5)
node[midway,below=5pt] {$ k$};

\end{tikzpicture}

%% file: fig2-1.tikz
\begin{tikzpicture}[scale=0.75]

\tikzset{
    vertex/.style={circle, fill=black, inner sep=1.6pt},
    blueedge/.style={blue, very thick},
    rededge/.style={red, very thick},
    lightrededge/.style={red, opacity=.25, thin}
}


\node[vertex,label=right:$v'_1$] (v1) at (0,0) {};
\node[vertex,label=right:$v'_2$] (v2) at (3,0) {};
\node[vertex,label=right:$v'_3$] (v3) at (6,0) {};
\node[vertex,label=right:$v'_{n-1}$] (vnm2) at (10,0) {};
\node[vertex,label=right:$v'_{n}$] (vnm1) at (13,0) {};

\node[newvertex,label=above:$x'_1$] (x1) at (1.5,2) {};
\node[newvertex,label=above:$x'_2$] (x2) at (4.5,2) {};
\node[newvertex,label=above:$x'_{n-1}$] (xnm2) at (11.5,2) {};
\node[newvertex,label=above:$x'_{n}$] (xnm1) at (14.5,2) {};

\draw[blueedge] (v1)--(x1)--(v2)--(x2)--(v3)--(6.75,1);
\draw[blueedge] (9.25,1)--(vnm2)--(xnm2)--(vnm1)--(xnm1);

\node[vertex] (a1) at (-0.6,-2.0) {};
\node[vertex] (a2) at (0.6,-2.0) {};

\foreach \x in {-0.55,-0.45,...,0.55}
    \draw[lightrededge] (v1) -- (\x,-2.0);

\draw[rededge] (v1)--(a1);
\draw[rededge] (v1)--(a2);

\draw[decorate,decoration={brace,amplitude=5pt}]
(0.6,-2.2) -- (-0.6,-2.2)
node[midway,below=5pt] {$d'_1$};

\node[vertex] (b1) at (2.4,-2.0) {};
\node[vertex] (b2) at (3.6,-2.0) {};

\foreach \x in {2.45,2.55,...,3.55}
    \draw[lightrededge] (v2) -- (\x,-2.0);

\draw[rededge] (v2)--(b1);
\draw[rededge] (v2)--(b2);

\draw[decorate,decoration={brace,amplitude=5pt}]
(3.6,-2.2) -- (2.4,-2.2)
node[midway,below=5pt] {$d'_2$};

\node[vertex] (c1) at (5.4,-2.0) {};
\node[vertex] (c2) at (6.6,-2.0) {};

\foreach \x in {5.45,5.55,...,6.55}
    \draw[lightrededge] (v3) -- (\x,-2.0);

\draw[rededge] (v3)--(c1);
\draw[rededge] (v3)--(c2);

\draw[decorate,decoration={brace,amplitude=5pt}]
(6.6,-2.2) -- (5.4,-2.2)
node[midway,below=5pt] {$d'_3$};

\node at (8,-2.0) {$\cdots$};

\node[vertex] (f1) at (9.4,-2.0) {};
\node[vertex] (f2) at (10.6,-2.0) {};

\foreach \x in {9.45,9.55,...,10.55}
    \draw[lightrededge] (vnm2) -- (\x,-2.0);

\draw[rededge] (vnm2)--(f1);
\draw[rededge] (vnm2)--(f2);

\draw[decorate,decoration={brace,amplitude=5pt}]
(10.6,-2.2) -- (9.4,-2.2)
node[midway,below=5pt] {$d'_{n-1}$};

\node[vertex] (d1) at (12.4,-2.0) {};
\node[vertex] (d2) at (13.6,-2.0) {};

\foreach \x in {12.45,12.55,...,13.55}
    \draw[lightrededge] (vnm1) -- (\x,-2.0);

\draw[rededge] (vnm1)--(d1);
\draw[rededge] (vnm1)--(d2);

\draw[decorate,decoration={brace,amplitude=5pt}]
(13.6,-2.2) -- (12.4,-2.2)
node[midway,below=5pt] {$d'_{n}$};

\draw[color =white, fill=white] (7,-1.5) rectangle (9,1.5);

\node at (8,0) {$\cdots$};
\node at (8,2) {$\cdots$};

\end{tikzpicture}

%% file: fig5.tikz
\begin{tikzpicture}[
    scale=0.76,
    transform shape,
    >=latex,
    every node/.style={align=center},
    box/.style={
        draw,
        rounded corners,
        thick,
        rectangle,
        inner sep=6pt,
        minimum width=4cm,
        minimum height=1.0cm
    },
    smallbox/.style={
        draw,
        rounded corners,
        thick,
        rectangle,
        inner sep=5pt,
        minimum width=3.6cm,
        minimum height=1.0cm
    },
    decision/.style={
        draw,
        rounded corners,
        thick,
        inner sep=5pt,
        minimum width=3cm,
        minimum height=1.2cm
    },
    line/.style={thick,->}
]

\node[box] (start) at (0,0) {\textbf{Apply Lemma \ref{lem:path}}\\[4pt]
Construct a blue copy of $P_n$\\
with average red degree $\mu$

};

\node[decision] (decide) at (0,-2.5) {\textbf{Cost:} $\displaystyle \left(1+\frac{\mu}{2}\right)n$};

\node[box] (left1) at (-4,-5.0) {
\textbf{Case 1: $\mu$ is large}

};

\node[smallbox] (left2) at (-4,-7.2) {
\textbf{Apply Lemma~\ref{lem:dub}}\\[4pt]
Extend directly to a blue copy of $P_{2n}$\\[8pt]
\textbf{Additional cost:} $\displaystyle (k+1-\mu)n$
};

\node[smallbox] (left3) at (-4,-9.5) {
\textbf{Total cost:}\\[8pt]
$\displaystyle 
\left(1+\frac{\mu}{2}\right)n+(k+1-\mu)n
$
};

\node[box] (right1) at (4,-5.0) {
\textbf{Case 2: $\mu$ is small}
};

\node[smallbox] (right2) at (4,-7.2) {
\textbf{Apply induction and Lemma \ref{lem:fus}}\\[4pt]
Construct a second blue copy of $P_n$ and merge.\\[8pt]
\textbf{Additional cost:} $\displaystyle \text{current inductive bound}$
};

\node[smallbox] (right3) at (4,-9.5) {
\textbf{Total cost:}\\[8pt]
$\displaystyle 
\left(1+\frac{\mu}{2}\right)n
+\text{current inductive bound}
$
};

\node[box] (balance) at (0,-12.3) {
\textbf{Choose a threshold for $\mu$ so that the two bounds agree:}\\[5pt]
$\displaystyle 
\left(1+\frac{\mu}{2}\right)n+(k+1-\mu)n
=
\left(1+\frac{\mu}{2}\right)n+\text{current inductive bound}
$
};

\node[box] (iterate) at (0,-15.1) {
{\textbf{Update coefficient and iterate}}
};

\draw[line] (start) -- (decide);
\draw[line] (decide) -- node[left,pos=0.55,xshift=-10pt] {$\mu>\alpha_m k$} (left1);
\draw[line] (decide) -- node[right,pos=0.55,xshift=10pt] {$\mu\le \alpha_m k$} (right1);

\draw[line] (left1) -- (left2);
\draw[line] (left2) -- (left3);

\draw[line] (right1) -- (right2);
\draw[line] (right2) -- (right3);

\draw[line] (left3) -- (balance);
\draw[line] (right3) -- (balance);
\draw[line] (balance) -- (iterate);

\end{tikzpicture}

%% file: concluding.tex
\section{Concluding remarks}

In this paper, we improved the asymptotic upper bounds for the online Ramsey
numbers of a fixed star versus a long path or cycle. In particular, for every
fixed positive integer $k$, we proved
\[
\tilde{r}(K_{1,k},P_n)
\le
\left(\frac{k+3}{3}\right)n+o(n),
\]
and, consequently,
\[
\tilde{r}(K_{1,k},C_n)
\le
\left(\frac{k+3}{3}\right)n+o(n).
\]
Together with the previously known lower bounds, this gives
\[
\left(\frac{k+3}{4}+o(1)\right)n
\le
\tilde{r}(K_{1,k},P_n),\,
\tilde{r}(K_{1,k},C_n)
\le
\left(\frac{k+3}{3}+o(1)\right)n.
\]

For each fixed $k$, Proposition~\ref{prop:limit} shows that the path and
cycle problems have the same asymptotic coefficient. It is therefore natural
to define
\[
L_k
=
\lim_{n\to\infty}
\frac{\tilde{r}(K_{1,k},P_n)}{n}
=
\lim_{n\to\infty}
\frac{\tilde{r}(K_{1,k},C_n)}{n}.
\]
Our results imply
\[
\frac{k+3}{4}
\le
L_k
\le
\frac{k+3}{3}.
\]
Determining $L_k$ for $k\ge 4$ remains open and interesting. In particular, it
would be worthwhile to determine whether the known lower bound is
asymptotically sharp for values of $k$ beyond $k=3$, or whether further
improvements to both the Builder and Painter strategies are possible.

Lastly, the result of Lemma \ref{lem:fus} can be readily generalized to any tree $T_n$ on $n$ vertices. Since this stronger result is not needed for our purposes, we omit its proof. Nevertheless, we note this generalization as it may be useful in future work on related problems.
